\documentclass[11pt]{amsart}
\usepackage{amsfonts, amssymb, amsmath, amsthm, euscript, color}
\usepackage{geometry}
\usepackage{graphicx}
\usepackage{amssymb}
\usepackage{epstopdf}
\usepackage{url}
\usepackage[all]{xy}

\swapnumbers

\theoremstyle{plain}
\newtheorem{thm}{Theorem}[section]
\newtheorem{lem}[thm]{Lemma}
\newtheorem{prop}[thm]{Proposition}
\newtheorem{cor}[thm]{Corollary}
\newtheorem{example}[thm]{Example}

\theoremstyle{definition}
\newtheorem*{Ack}{Acknowledgement}
\newtheorem*{con}{Convention}
\newtheorem*{thmA}{Theorem A}
\newtheorem*{corB}{Corollary B}
\newtheorem*{queC}{Question C}
\newtheorem{deff}[thm]{Definition}

\newtheorem{remark}[thm]{Remark}

\theoremstyle{remark}

\newcommand*{\e}{\ensuremath{\varepsilon}}

\newcommand*{\Ker}{\ensuremath{\text{\upshape Ker}}}

\newcommand*{\ad}{\ensuremath{\text{\upshape ad}}}

\newcommand*{\Chara}{\ensuremath{\text{\upshape char}}}

\newcommand*{\Grp}{\ensuremath{\text{\upshape G}}}
\newcommand*{\PS}{\Gamma}

\def\dim{\operatorname{dim}}

\begin{document}
\thispagestyle{empty}

\title{Locality criteria for cocommutative Hopf algebras}

\author{Xingting Wang}

\address{
Department of Mathematics\\ University of Washington\\ Seattle, WA 98195\\Department of Mathematics}

\email{xingting@uw.edu}



\begin{abstract}
We prove that a finite-dimensional cocommutative Hopf algebra $H$ is local, if and only if the subalgebra generated by the first term of its coradical filtration $H_1$ is local. In particular if $H$ is connected, $H$ is local if and only if all the primitive elements of $H$ are nilpotent. 
\end{abstract}
\maketitle

\section{Introduction}
Assume the base field is algebraically closed. The structure theorem for cocommutative Hopf algebras is due to Milnor, Moore, Cartier and Kostant around 1963 \cite[Theorem 1.1]{andruskiewitsch2000finite}. In characteristic zero, it says that any cocommutative Hopf algebra is a smash product of the universal enveloping algebra of a Lie algebra and a group algebra. The statement becomes less definite in positive characteristic,  where the universal enveloping algebra of a Lie algebra is replaced by any connected cocommutative Hopf algebra (see \cite[\S 5.6]{montgomery1993hopf}).

Preliminary results are given in Section 2. In this paper, we investigate one ring-theoretic property for finite-dimensional cocommutative Hopf algebras, i.e., the criteria for locality. Let $H$ be any Hopf algebra. We say that $H$ is local as an algebra if it has a unique maximal (one-sided) ideal. Also denote by $H_1$ the first term of the coradical filtration of $H$. We will prove the following main result:

\begin{thmA} Let $H$ be a finite-dimensional cocommutative Hopf algebra over any arbitrary field $k$. Then $H$ is local if and only if the subalgebra generated by $H_1$ is local.
\end{thmA}

We introduce the lower (resp., upper) power series for commutative (resp., cocommutative) Hopf algebras in positive characteristic in Section \ref{ULPS}. Our approach is based on the duality theorem for those two series proved in Section \ref{Duality}. By using it, we verify a special case for Theorem A when the cocommutative Hopf algebra is connected in Section \ref{connected}. The Cartier-Kostant-Milnor-Moore theorem is reviewed in Section \ref{PCHA} and the complete proof of Theorem A is given in Section \ref{TMT}. We also provide the following corollary in Section \ref{TMT}:

\begin{corB}
Let $H$ be a finite-dimensional connected Hopf algebra. Then the following are equivalent:
\begin{itemize}
\item[(i)] $H$ is local.
\item[(ii)] $u(\mathfrak g)$ is local, where $\mathfrak g$ is the primitive space of $H$.
\item[(iii)] All the primitive elements of $H$ are nilpotent.
\end{itemize}
\end{corB}

This corollary uses Engel's theorem in the representation theory of Lie algebras. Hence we can view our result as a generalization of Engel's theorem in the category of cocommutative Hopf algebras of finite dimension. In the last section, two examples are provided in order to show that the cocommutativity condition is necessary in Theorem A. At last, we ask the following natural question.
\begin{queC} How to generalize this locality criteria to infinite-dimensional cocommutative Hopf algebras?
\end{queC}

\begin{Ack}
The author is grateful to Professor James Zhang, Guangbin Zhang and Cris Negron for their valuable suggestions to this paper. The research was partially supported by the US National Science Foundation [DMS-0855743].
\end{Ack}

\section{Preliminary results}
Throughout we work over a field $k$ and $\otimes$ denotes $\otimes_k$. For a finite-dimensional $k$-linear space $V$, identify $V$ with its double dual $V^{**}$ through the natural isomorphism. The standard collection $(H,m,u,\Delta,\e,S)$ is used to denote a Hopf algebra. We first recall some basic definitions and facts regarding $H$.

\begin{deff}\cite[Definitions 5.1.5, 5.2.1]{montgomery1993hopf}
The \emph{coradical} $H_0$ of $H$ is the sum of all simple subcoalgebras of $H$. The Hopf algebra $H$ is \emph{pointed} if every simple subcoalgebra is one-dimensional, and $H$ is \emph{connected} if $H_0$ is one-dimensional. For each $n\geq 1$, inductively set  \[H_n=\Delta^{-1}(H\otimes H_{n-1} + H_0\otimes H).\] The chain of subcoalgebras $H_0\subseteq H_1 \subseteq \ldots \subseteq H_{n-1}\subseteq H_n \subseteq\ldots $ is the \emph{coradical filtration} of $H$. 
\end{deff}

\begin{lem}\label{2duality}
Let $H$ be a finite-dimensional Hopf algebra.
\begin{itemize}
\item[(i)] $H^*$ has a natural Hopf algebra structure and $H=H^{**}$ under the natural isomorphism.
\item[(ii)] Denote $J$ as the Jacobson radical of $H^*$. Then $H_n=(H^*/J^{n+1})^*$ for any $n\ge 0$. 
\item[(iii)] $H$ is local if and only if its augmentation ideal $H^+=\Ker \e$ is nilpotent if and only if $H^*$ is connected.
\item[(iv)] Let $L$ be a Hopf subalgebra of $H$. Then $L_n=L\cap H_n\subseteq H_n$ for all $n\ge 0$. Moreover, if $H$ is connected, so is $L$.
\item[(v)] The antipode $S$ is bijective and $S(H_n)=H_n$ for all $n\ge 0$.
\end{itemize}
\end{lem}
\begin{proof}
(i) is well-known, see \cite[Theorem 9.1.3]{montgomery1993hopf}.

(ii) comes from \cite[Proposition 5.2.9(2)]{montgomery1993hopf}.

(iii) First suppose that $H$ is local. Then it is clear that the Jacobson radical of $H$ is the unique maximal ideal of $H$. Since $H$ already has a maximal ideal, i.e., the augmentation ideal $H^+$, the Jacobson radical of $H$ equals $H^+$. It follows that $H^+$ is nilpotent for the Jacobson radical of any Artinian ring is nilpotent. Conversely, suppose that $H^+$ is nilpotent. Then the Jacobson radical of $H$ contains $H^+$ hence it equals $H^+$ for $H^+$ is already maximal. Finally, the fact that $H$ is local if and only if $H^*$ is connected follows from \cite[Remark 5.1.7]{montgomery1993hopf}.

(iv) is a consequence of \cite[Lemma 5.2.12]{montgomery1993hopf}.

(v) By \cite[Corollary 5.1.6(3)]{S}, $S$ is bijective. Moreover, $S$ is both an anti-coalgebra and anti-algebra map by \cite[Proposition 4.0.1]{S}. Since the opposite coalgebra of a simple coalgebra is still simple, it is clear that $S(H_0)=H_0$. In the definition of coradical filtration, it is the same to define $H_n=\Delta^{-1}(H\otimes H_0+H_{n-1}\otimes H)$ for each $n\ge 1$. Then by induction, it is easy to show that $S(H_n)=H_n$ for all $n\ge 0$.
\end{proof}

\begin{deff}\cite[Definitions 3.4.1, 3.4.5]{montgomery1993hopf}
A Hopf subalgebra $K$ of $H$ is \emph{normal} if both
\begin{eqnarray*}
\sum (Sh_1)kh_2\subseteq K\ \text{and}\ \sum h_1k(Sh_2)\subseteq K,
\end{eqnarray*}
for all $k\in K,h\in H$. A Hopf ideal $I$ of $H$ is \emph{normal} if both
\begin{eqnarray*}
\sum h_1Sh_3\otimes h_2\subseteq H\otimes I\ \text{and}\ \sum h_2\otimes (Sh_1)h_3\subseteq I\otimes H,
\end{eqnarray*} 
for all $h\in I$. 
\end{deff}

\begin{remark}\label{1R}
Let $H$ be any Hopf algebra with Hopf subalgebra $K$ and Hopf ideal $I$. 
\begin{itemize}
\item[(a)] If $K$ is normal, then $K^+H=HK^+$ is a Hopf ideal of $H$. 
\item[(b)] For finite-dimensional $H$, the Hopf ideal $I$ is normal if and only if $(H/I)^*$ is a normal Hopf subalgebra of $H^*$.
\end{itemize}
\end{remark}

\begin{con}
Throughout, when $K$ is a normal Hopf subalgebra of $H$, we use $H/K$ to denote the quotient Hopf algebra $H/HK^+$.
\end{con}

\begin{lem}\label{dim}
Suppose $\Chara k=p>0$. Let $K\subseteq H$ be finite-dimensional connected Hopf algebras. Then 
\begin{eqnarray*}
\dim H/\dim K\ge p^{\dim H_n/K_n}, 
\end{eqnarray*}
for any $n$ such that $K_i=H_i$ for all $i< n$.
\end{lem}
\begin{proof}
It follows from \cite[lemma 4.1]{wang2012connected}, the inequality holds for the minimal integer $n$ such that $K_n\neq H_n$. Then it is easy to check the statement.
\end{proof}

\begin{remark}\label{2R}
Let $K\subseteq H$ be finite-dimensional Hopf algebras. By \cite[pp. 290]{normalbasis}, $H$ is isomorphic to $K\otimes H/K^+H$ as left $K$-modules and as right $H/K^+H$-comodules. As a consequence, $\dim(H/K)=\dim H/\dim K$ provided that $K$ is normal.
\end{remark}

\begin{lem}\label{ndual}
Let $K\subseteq L$ be normal Hopf subalgebras of a finite-dimensional Hopf algebra $H$. Then $(H/K)^*/(H/L)^*=(L/K)^*$. In particular, $H^*/(H/K)^*=K^*$ for any normal Hopf subalgebra $K$.
\end{lem}
\begin{proof}
The result follows from \cite[lemma 5.1]{wang2012connected}. 
\end{proof}

\begin{lem}\label{ILI}
Let $H$ be a finite-dimensional Hopf algebra, and $E\supseteq k$ be a field extension. Then
\begin{itemize}
\item[(i)] $H$ is local if and only if any subalgebra of $H$ is local.
\item[(ii)] $H$ is local if and only if $H\otimes E$ is local. 
\item[(iii)] $(H\otimes E)_n\subseteq H_n\otimes E$ for all $n\ge 0$
\end{itemize}
\end{lem}
\begin{proof}
$(\text{i})$ Suppose that $H$ is local and $A$ is any subalgebra of $H$. Denote $J$ as the Jacobson radical of $A$ and $I=H^+\cap A$. By Lemma \ref{2duality}(iii), $H^+$ is nilpotent and so is $I\subseteq H^+$. Hence $I\subseteq J$. Moreover, $J\subseteq I$ for $I$ has codimension one in $A$. Therefore $J=I$ and $A$ is local. The inverse is trivial.

$(\text{ii})$ It is easy to see that $H\otimes E$ is a Hopf algebra \cite[pp. 21]{montgomery1993hopf} and $(H\otimes E)^+=H^+\otimes E$. Suppose that $H$ is local. Then $(H\otimes E)^+$ is nilpotent for $H^+$ is. Hence $H\otimes E$ is local by Lemma \ref{2duality}(iii). Assume $H\otimes E$ to be local. Therefore $H^+=H^+\otimes 1\subseteq H^+\otimes E$ is a nilpotent ideal of $H$. So $H$ is local.

$(\text{iii})$ Denote $J$ and $J'$ as the Jacobson radicals of $H^*$ and $H^*\otimes E$. By Lemma \ref{2duality}(ii), the assertion is equivalent to that there is a surjection from $(H^*/J^n)\otimes E$ to $(H^*\otimes E)/J'^n$ for all $n\ge 0$. It is true because $J\subseteq J'$ for $J$ is nilpotent.
\end{proof}

\section{Upper and lower power series}\label{ULPS}
Throughout this section assume that $\Chara\ k=p>0$.

\begin{deff}
Define the \emph{lower power series} of a commutative Hopf algebra $H$ as: 
\begin{eqnarray*}
\PS_n(H)=\{h^{p^n}|h\in H\},
\end{eqnarray*}
for all $n\ge 0$. Define the \emph{upper power series} of a cocommutative Hopf algebra $H$ as: 
\begin{eqnarray*}
\PS^0(H)=k,\ \text{and}\ \PS^n(H)=k\langle H_{p^{n-1}}\rangle
\end{eqnarray*}
for $n\ge 1$, where $k\langle H_{p^{n-1}}\rangle $ denotes the subalgebra generated by the $p^{n-1}$-th term of its coradical filtration.
\end{deff}

\begin{prop}\label{nseries}
Regarding the above definitions, we have
\begin{itemize}
\item[(i)]  The lower power series of a commutative Hopf algebra is a descending chain of normal Hopf subalgebras.  
\item[(ii)] The upper power series of a cocommutative Hopf algebra is an increasing chain of Hopf subalgebras.
\end{itemize}
\end{prop}
\begin{proof}
(i) Let $H$ be a commutative Hopf algebra. Since the base field $k$ has characteristic $p>0$, it is clear that each term of the lower power series of $H$ is a subalgebra of $H$. Moreover, we have
\begin{align*}
\Delta\left(h^{p^n}\right)=\Delta(h)^{p^n}=\left(\sum h_1\otimes h_2\right)^{p^n}=\sum h_1^{p^n}\otimes h_2^{p^n},\quad S\left(h^{p^n}\right)=S(h)^{p^n}
\end{align*}
for all $h\in H$ and $n\ge 0$. Hence $\{\PS_n(H)\}$ is a descending chain of Hopf subalgebras of $H$ and the normality follows from the commutativity of $H$.

(ii) Let $H$ be a cocommutative Hopf algebra. In the upper power series of $H$, each term $\PS^n(H)$ is already a sub-bialgebra since it is generated by the subcoalgebra $H_{p^{n-1}}$ of $H$. Hence we only need to show that $S(\PS^n(H))\subseteq \PS^n(H)$. By \cite[Proposition 4.0.1]{S}, $S$ is  an anti-algebra map. Hence it suffices to prove that $S(H_{p^{n-1}})=H_{p^{n-1}}$, which is Lemma \ref{2duality}(v).  
\end{proof}

\begin{remark}\label{normal}
Let $H$ be a finite-dimensional cocommutative Hopf algebra. All $\PS_n(H^*)^+H^*$ are normal Hopf ideals of $H^*$. Moreover, $\left(H^*/\PS_n(H^*)^+H^*\right)^*$ are normal Hopf subalgebras of $H$ by Remark \ref{1R}(b).
\end{remark}

\section{Duality}\label{Duality}
In the following two sections, assume that $k$ is algebraically closed with $\Chara\ k=p>0$. We want to prove the following duality theorem about lower and upper power series:

\begin{thm}\label{MDT}
Let $H$ be a finite-dimensional cocommutative connected Hopf algebra. Then 
\begin{eqnarray*}
\PS^n(H)=(H^*/\PS_n(H^*))^*\ \text{and}\ \PS_n(H^*)=(H/\PS^n(H))^*.
\end{eqnarray*}
\end{thm}

First, we fix some notations in this section. Let $H$ be a finite-dimensional cocommutative connected Hopf algebra. Therefore $H^*$ is commutative local by Lemma \ref{2duality}(iii). We will write $\PS^n$ for the upper power series of $H$ and $\PS_n$ for the lower power series of $H^*$. Denote $J_n$ as the Jacobson radical of $\PS_n$ for each $n$, where we write $J=J_0$ for the Jacobson radical of $H^*$. Since $\PS_n$ is local by Lemma \ref{ILI}(i), it is clear that $J_n=\PS_n^+$ by Lemma \ref{2duality}(iii) and $H^*/\PS_n=H^*/J_nH^*$ for all $n\ge 0$. 
\begin{lem}\label{1L}
We have $J_n\subseteq J^{p^{n}}$ and $J_n\cap J^{p^{n}+1}=J_n^{2}$. Moreover,
\begin{eqnarray*}
\dim \left(\PS_{n-1}/\PS_n\right)=p^{\dim\left(J_{n-1}/(J^{p^{n-1}+1}\cap J_{n-1})\right)},
\end{eqnarray*}
for any $n\ge 1$.
\end{lem}
\begin{proof}
According to \cite[Theorem 14.4]{waterhouse1979introduction}, we can write
\begin{eqnarray*}
H^*=k\left[x_1,x_2,\cdots,x_d\right]\Big/\left(x_1^{p^{e_1}},x_2^{p^{e_2}},\cdots,x_d^{p^{e_d}}\right)
\end{eqnarray*}
as an algebra. Therefore each term $\PS_n$ in the lower power series of $H^*$ is generated by $\left\{x_i^{p^n}\big|1\le i\le d\right\}$ and
\begin{eqnarray*}
J_n=\left(x_1^{p^n},x_2^{p^n},\cdots,x_d^{p^n}\right).
\end{eqnarray*}
It follows that $J_n\subseteq J^{p^n}$. Hence $J_n^2=J_nJ_n\subseteq J^{p^n}J=J^{p^n+1}$ and $J_n^2\subseteq J_n\cap J^{p^{n}+1}$. Then in order to show $J_n\cap J^{p^{n}+1}=J_n^{2}$, it suffices to prove that $\left(J_n\cap J^{p^{n}+1}\right)/J_n^2=0$ . It is clear that any element in $\left(J_n\cap J^{p^{n}+1}\right)/J_n^2$ can be represented by $\sum \lambda_ix_i^{p^n}$ for $\lambda_i\in k$. Without loss of generality, we can assume that all $x_i^{p^n}\neq 0$. Observe that any term $\sum \alpha_ix_i^{d_i}\in J^{p^n+1}$ must have $d_i\ge p^n+1$ or $\alpha_i=0$. This implies that all $\lambda_i=0$, which completes the proof. 

By the description of each $\PS_{n}$, there exists some integer $l$ such that
\begin{eqnarray*}
\PS_{n-1}/\PS_{n}\cong k\left[y_1,\cdots,y_l\right]\big/\left(y_1^p,\cdots,y_l^p\right)
\end{eqnarray*} 
as algebras. Moreover, $l=\dim J_{n-1}/J_{n-1}^2$. Therefore 
\begin{eqnarray*}
\dim\left(\PS_{n-1}/\PS_{n}\right)=p^l=p^{\dim\left(J_{n-1}/J_{n-1}^2\right)}=p^{\dim\left(J_{n-1}/(J^{p^{n-1}+1}\cap J_{n-1})\right)}.
\end{eqnarray*}
\end{proof}

\begin{lem}\label{2L}
We have for each $n\ge 1$
\begin{eqnarray*}
p^{\dim\left((J_{n-1}H^*+J^{p^{n-1}+1})/J^{p^{n-1}+1}\right)}=\dim (H^*/\PS_{n})^*/\dim(H^*/\PS_{n-1})^*.
\end{eqnarray*}
\end{lem}
\begin{proof}
We know $\PS_{n}\subseteq \PS_{n-1}\subseteq H^*$ are normal Hopf subalgebras of $H^*$ by Proposition \ref{nseries}(i). Applying the isomorphism in Lemma \ref{ndual}, we have 
\begin{eqnarray*}
\dim \left((H^*/\PS_{n})^*/(H^*/\PS_{n-1})^*\right)=\dim\left(\PS_{n-1}/\PS_{n}\right).
\end{eqnarray*}
Moreover, $\left(H^*/\PS_{n-1}\right)^*$ is normal in $H$ hence in $\left(H^*/\PS_{n}\right)^*$ by Remark \ref{normal}. Then by Remark \ref{2R}, we have 
\begin{eqnarray*}
\dim (H^*/\PS_{n})^*/\dim(H^*/\PS_{n-1})^*=\dim \left((H^*/\PS_{n})^*/(H^*/\PS_{n-1})^*\right)=\dim\left(\PS_{n-1}/\PS_{n}\right)
\end{eqnarray*}
According to Lemma \ref{1L}, $J_{n-1}J\subseteq J^{p^{n-1}}J=J^{p^{n-1}+1}$. Hence the natural map from $J_{n-1}$ to $(J_{n-1}H^*+J^{p^{n-1}+1})/J^{p^{n-1}+1}$ is surjective, which has kernel $J_{n-1}\cap J^{p^{n-1}+1}$. Then the result follows by Lemma \ref{1L}.
\end{proof}

\begin{lem}\label{3L}
We have $H_{p^n-1}\subseteq \left(H^*/\PS_n\right)^*$. Moreover, $\PS^n\subseteq  \left(H^*/\PS_n\right)^*$ for all $n\ge 0$. 
\end{lem}
\begin{proof}
There is a natural surjection $H^*/\PS_n=H^*/(J_nH^*)\twoheadrightarrow H^*/J^{p^{n}}$ since  $J_n\subseteq J^{p^n}$ by Lemma \ref{1L}. Notice that the dual of the last term is $H_{p^{n}-1}$ in $H^{**}=H$ by Lemma \ref{2duality}(ii). Therefore we have $H_{p^n-1}\subseteq \left(H^*/\PS_n\right)^*$ by taking the dual of the above surjection. Because $p^{n-1}\le p^n-1$, $\PS^n(H)=k\langle H_{p^{n-1}}\rangle\subseteq k\langle H_{p^{n}-1}\rangle\subseteq \left(H^*/\PS_n\right)^*$.
\end{proof}

\noindent
\bf{Proof of Theorem \ref{MDT}}\rm.  First of all, we prove  $\PS^n=(H^*/\PS_n)^*$ by induction on $n$. When $n=0$, both sides are the base field $k$. Next, assume that the statement is true for $n-1$. Write $C=\PS^{n-1}$ and $D=\PS^n$, where $C=(H^*/\PS_{n-1})^*$ by induction. Choose any integer $m\le p^{n-1}-1$. By Lemma \ref{3L}, $H_m\subseteq C=(H^*/\PS_{n-1})^*$. Moreover by Lemma \ref{2duality}(iv), $H_m\subseteq C_m=C\cap H_m\subseteq D_m\subseteq H_m$, which implies that $C_m=D_m=H_m$. Then we have
\begin{eqnarray*}
\dim D/\dim C\ge p^{\dim\left(D_{p^{n-1}}/C_{p^{n-1}}\right)}
\end{eqnarray*}
by Lemma \ref{dim}. According to Lemma \ref{2duality}(ii), a simple calculation yields that
\begin{align*}
\dim C_{p^{n-1}}=\dim \frac{H^*/(J_{n-1}H^*)}{\left(J/J_{n-1}H^*\right)^{p^{n-1}+1}}=\dim \frac{H^*}{J_{n-1}H^*+J^{p^{n-1}+1}}.
\end{align*}
Since $D$ is generated by $H_{p^{n-1}}$, it is clear that $D_{p^{n-1}}=H_{p^{n-1}}$. Thus
\begin{eqnarray*}
\dim\frac{D_{p^{n-1}}}{C_{p^{n-1}}}&=&\dim\frac{H_{p^{n-1}}}{C_{p^{n-1}}}=\dim\frac{H^*/J^{p^{n-1}+1}}{H^*/(J_{n-1}H^*+J^{p^{n-1}+1})}=\dim\frac{J_{n-1}H^*+J^{p^{n-1}+1}}{J^{p^{n-1}+1}}.
\end{eqnarray*}
Therefore by Lemma \ref{2L} and Lemma \ref{3L} , we have
\begin{eqnarray*}
\dim (H^*/\PS_n)^*\big/\dim (H^*/\PS_{n-1})^*&\ge&\dim\PS^n/\dim \PS^{n-1}=\dim D/\dim C\\&\ge& p^{\dim\left(D_{p^{n-1}}/C_{p^{n-1}}\right)}
=\dim (H^*/\PS_n)^*/\dim (H^*/\PS_{n-1})^*.
\end{eqnarray*}
Hence $\dim \PS^n=\dim (H^*/\PS_n)^*$. By Lemma \ref{3L}, we know $\PS^n\subseteq(H^*/\PS_n)^*$. So $\PS^n=(H^*/\PS_n)^*$. The other statement is checked as below. By Lemma \ref{ndual}, 
\begin{eqnarray*}
\left(H/\PS^n\right)^*=\left(H^{**}/(H^*/\PS_n)^*\right)^*=(\PS_n)^{**}=\PS_n.
\end{eqnarray*}
\begin{cor}\label{factor}
For the factors, we have 
\begin{eqnarray*}
\PS^n/\PS^m=(\PS_m/\PS_n)^*
\end{eqnarray*}
for any $n\ge m$.
\end{cor}
\begin{proof}
By Lemma \ref{ndual}, whenever $n\ge m$, we have
\begin{eqnarray*}
\PS^n/\PS^m=(H/\PS_n)^*/(H/\PS_m)^*=(\PS_m/\PS_n)^*.
\end{eqnarray*}
\end{proof}
\begin{cor}\label{CMDT}
The upper power series of $H$ is a sequence of normal Hopf subalgebras. 
\end{cor}
\begin{proof}
It follows from Theorem \ref{MDT} and Remark \ref{normal}.
\end{proof}

\section{Finite-Dimensional Cocommutative Connected Hopf Algebras}\label{connected}
We still assume $k$ to be algebraically closed of characteristic $p>0$. Recall that a \emph{coalgebra filtration} of a coalgebra $C$ is any set of subspaces  $\{A_n\}_{n\ge 0}$ satisfying the two conditions below: (i) $A_{n}\subseteq A_{n+1},\ C=\cup_{n\ge 0} A_n$ (ii) $\Delta A_n\subseteq \sum_{i=0}^{n} A_{i}\otimes A_{n-i}$. We state the following two dualized properties:
\begin{prop}\label{cocommutativeconnectedHopfalgebra}
Let $H$ be a finite-dimensional cocommutative connected Hopf algebra with upper power series $\{\PS^n(H)\}$. Then the following are equivalent:
\begin{itemize}
\item[(i)] $\PS^1(H)$ is local.
\item[(ii)] $\PS^n(H)/\PS^{n-1}(H)$ is local for all $n\ge 1$.
\item[(iii)] $H$ is local.
\end{itemize}
\end{prop}
\begin{prop}\label{commutativeconnectedHopfalgebra}
Let $K=H^*$ be a finite-dimensional commutative local Hopf algebra with lower power series $\{\PS_n(K)\}$. Then the following are equivalent:
\begin{itemize}
\item[(i)] $K/\PS_1(K)$ is connected.
\item[(ii)] $\PS_{n-1}(K)/\PS_{n}(K)$ is connected for all $n\ge 1$.
\item[(iii)] $K$ is connected.
\end{itemize}
\end{prop}

\noindent
\bf{Proof of Propositions \ref{cocommutativeconnectedHopfalgebra} and \ref{commutativeconnectedHopfalgebra}}\rm. By Corollary \ref{factor} and Lemma \ref{2duality}(iii), statements (i), (ii), (iii) of each proposition are equivalent, respectively. Hence we can prove them together. The strategy is to show $(\text{i})\Rightarrow(\text{ii})$ in Proposition \ref{commutativeconnectedHopfalgebra} and $(\text{ii})\Rightarrow(\text{iii})$ with $(\text{iii})\Rightarrow(\text{i})$ in Proposition \ref{cocommutativeconnectedHopfalgebra}.

$(\text{i})\Rightarrow (\text{ii})$ in Proposition \ref{commutativeconnectedHopfalgebra}. Write $C=K/\PS_1(K)$ and $D=\PS_{n-1}(K)/\PS_{n}(K)$ for any $n\ge 1$. Define map $f: K\to \PS_{n-1}(K)$ as $f(x)=x^{p^{n-1}}$ for all $x\in K$. It is clear that $f$ is surjective and it induces a surjection from $C$ to $D$, which we still denote as $f$. Notice that $f$ is not $k$-linear, but it is easy to show that that $\{f(C_n)\}_{n\ge 0}$ is a coalgebra filtration of $D$. Because $C_n$ is the coradical filtration of $C$, by  \cite[Theorem 5.2.2]{montgomery1993hopf}, $\{C_n\}_{n\ge 0}$ exhausts $C$. Hence $D=\bigcup_{n\ge 0}f(C_n)$. Moreover since $K$ is commutative:
\begin{eqnarray*}
\Delta\left(f(C_n)\right)&=&\left(\Delta(C_n)\right)^{p^{n-1}}\subseteq\left(\sum_{i=0}^{n}C_{i}\otimes C_{n-i}\right)^{p^{n-1}}\subseteq\sum_{i=0}^{n}C_{i}^{p^{n-1}}\otimes C_{n-i}^{p^{n-1}}\\ 
&&\subseteq\sum_{i=0}^{n}f\left(C_{i}\right)\otimes f\left(C_{n-i}\right)
\end{eqnarray*}
for all $n\ge 0$. By \cite[Lemma 5.3.4]{montgomery1993hopf}, $D_0\subseteq f(C_0)$ and the result follows. 

$(\text{ii})\Rightarrow(\text{iii})$ in Proposition \ref{cocommutativeconnectedHopfalgebra}. Because $H$ is finite-dimensional, there exists some integer $d$ such that $\PS^d(H)=H$. Since each factor $\PS^{n}(H)/\PS^{n-1}(H)$ is local for all $1\le n\le d$, there are integers $s_{n}$ such that $\left(\PS^{n}(H)^{+}\right)^{s_{n}}\subseteq \PS^{n}(H)\PS^{n-1}(H)^+$ for all $1\le n\le d$ by Lemma \ref{2duality}(iii). By Corollary \ref{CMDT}, $\PS^n(H)$ is normal in $\PS^{n+1}(H)$ for any $n\ge 0$. Hence by Remark \ref{1R}(a), we have
\begin{align*}
\left(\PS^{n+1}(H)\PS^n(H)^+\right)^t=\PS^{n+1}(H)^t\left(\PS^n(H)^+\right)^t\subseteq H\left(\PS^{n}(H)^+\right)^t
\end{align*}
for all $t\ge 0$. Denote $s=\prod s_n$, then 
\begin{align*}
&(H^+)^s=(\PS^d(H)^{+})^{s}=\left\{\left(\PS^d(H)^+\right)^{s_d}\right\}^{\frac{s}{s_d}}\subseteq \left(\PS^d(H)\PS^{d-1}(H)^+\right)^{\frac{s}{s_d}}\subseteq H\left(\PS^{d-1}(H)^+\right)^{\frac{s}{s_d}}\\
&=H\left\{\left(\PS^{d-1}(H)^+\right)^{s_{d-1}}\right\}^{\frac{s}{s_ds_{d-1}}}\subseteq H\left(\PS^{d-1}(H)\PS^{d-2}(H)^+\right)^{\frac{s}{s_ds_{d-1}}}\subseteq H\left(\PS^{d-2}(H)^+\right)^{\frac{s}{s_ds_{d-1}}}\\
&\subseteq\cdots\subseteq H\left(\PS^1(H)^+\right)^{s_1}=0.
\end{align*}
Thus $H$ is local by Lemma \ref{2duality}(iii). 

$(\text{iii})\Rightarrow (\text{i})$ in Proposition \ref{cocommutativeconnectedHopfalgebra}. This holds by Lemma \ref{ILI}(i).

\section{Pointed Cocommutative Hopf Algebras}\label{PCHA}
Let $H$ be any Hopf algebra, and $G$ denote the group of group-like elements of $H$. Recall \cite[Definition 5.6.1]{montgomery1993hopf} that any \emph{irreducible component} of $H$ is a maximal subcoalgebra, where any two of its nonzero subcoalgebras have nonzero intersection.  We use $H_g$ to denote the irreducible component  of $H$ containing $g$ for every $g\in G$.

\begin{prop} \cite[Corollary 5.6.4]{montgomery1993hopf}\label{MST}
\begin{itemize}
\item[(i)] $H_xH_y=H_{xy}$ and $SH_{x}\subseteq H_{x^{-1}}$ for all $x,y\in G$. 
\item[(ii)] $H_e$ is a Hopf subalgebra of $H$, where $e$ is the identity of $G$.
\item[(iii)] If $H$ is pointed cocommutative, then $H_e\#kG\cong H$ via $h\#x\mapsto hx$.
\end{itemize}
\end{prop}

\begin{lem}\label{lemstr}
Let $H$ be a finite-dimensional pointed cocommutative Hopf algebra. Then
\begin{itemize}
\item[(i)] $H_e$ is a connected normal Hopf subalgebra of $H$.
\item[(ii)] $H/H_e^+H\cong k[G]$.
\item[(iii)] $H$ is local if and only if $H_e$ and $k[G]$ are local.
\end{itemize}
\end{lem}
\begin{proof}
$(\text{i})$ By Proposition \ref{MST}(iii), we know $H$ is generated by $H_e$ and $k[G]$. We only need to check that for every $g\in G$, $gH_eg^{-1}\subseteq H_e$. It is true since $gH_eg^{-1}\subseteq H_{g}H_eH_{g^{-1}}=H_{geg^{-1}}=H_{e}$. We know $H_e$ is connected for it contains only one simple subcoalgebra $ke$. 

$(\text{ii})$ It follows from Proposition \ref{MST} since $H\cong H_e\#kG$.

$(\text{iii})$ Suppose that $H_e$ and $k[G]$ are local. Then by Lemma \ref{2duality}(iii) there exist integers $n,m$ such that $(H_e^+)^{n}=0$ and $((G-1)k[G])^m=0$. By part $(\text{ii})$, we have $(H^+)^m\subseteq H_e^+H$. Hence $(H^+)^{nm}\subseteq (H_e^+H)^n=0$, which implies that $H$ is local again by Lemma \ref{2duality}(iii). The inverse direction holds by Lemma \ref{ILI}(i).
\end{proof}

\section{The Main Theorem}\label{TMT}
Let $H$ be a finite-dimensional Hopf algebra over any arbitrary field $k$.  We still use $\PS^1(H)$ to denote the subalgebra of $H$ generated by $H_1$. We can use the argument in the proof of Proposition \ref{nseries}(ii) to show that $\PS^1(H)$ is also a Hopf subalgebra of $H$.
\begin{thm}\label{TMR}
Let $H$ be a finite-dimensional cocommutative Hopf algebra over any arbitrary field $k$. Then $H$ is local if and only if the subalgebra generated by $H_1$ is local.
\end{thm}
\begin{proof}
Denote $K=H\otimes \bar{k}$. According to Lemma \ref{ILI}(ii), $H$ is local if and only if $K$ is local. Moreover by Lemma \ref{ILI}(iii), $\PS^1(K)\subseteq \PS^1(H)\otimes \bar{k}$ since $K_1\subseteq H_1\otimes \bar{k}$. Hence without loss of generality, we can assume $k$ to be algebraically closed. Thus $H$ is pointed by \cite[pp. 76]{montgomery1993hopf}. 

If $k$ has characteristic zero, $H$ must be a group algebra by \cite[Theorem 1.1]{andruskiewitsch2000finite}. Hence $H$ is local if and only if it is just the base field $k$. The statement is trivial in this case. On the other hand, assume that $\Chara\ k\not= 0$. Denote $A=H_e$ and $G=\Grp(H)$ as in Proposition \ref{MST}. Since $k[G]\subseteq H_1\subseteq \PS^1(H)$, it is local by Lemma \ref{ILI}(i). So is $\PS^1(A)$ since $A_1\subseteq H_1$ by Lemma \ref{2duality}(iv). By Proposition \ref{cocommutativeconnectedHopfalgebra}, $A$ is local for it is cocommutative and connected. Finally, the proof is complete by Lemma \ref{lemstr}(iii).  
\end{proof}

\begin{cor}
Let $H$ be a finite-dimensional connected Hopf algebra. Then the following are equivalent:
\begin{itemize}
\item[(i)] $H$ is local.
\item[(ii)] $u(\mathfrak g)$ is local, where $\mathfrak g$ is the primitive space of $H$.
\item[(iii)] All the primitive elements of $H$ are nilpotent.
\end{itemize}
\end{cor}
\begin{proof}
In the connected case, the base field $k$ has positive characteristic \cite[Remark 2.7]{LinXT} and $\PS^1(H)$ is the restricted enveloping algebra of the primitive space $\mathfrak g$. Hence the equivalence between $(\text{i})$ and $(\text{ii})$ follows from Theorem \ref{TMR}. 

$(\text{ii})\Rightarrow (\text{iii})$ is clear since $u(\mathfrak g)$ contains all the primitive elements of $H$ and its augmentation ideal is nilpotent. 

$(\text{iii})\Rightarrow (\text{ii})$. In characteristic $p$, $\mathfrak g$ is a restricted Lie algebra. $(\text{iii})$ asserts that all elements $x\in \mathfrak g$ are nilpotent. Then $(\ad x)^{p^n}=\ad(x^{p^n})=0$ for sufficiently larger $n$. By Engel's Theorem \cite[\S 3.2]{humphreys1980introduction}, $\mathfrak g$ is nilpotent. Any representation of $u(\mathfrak g)$ is a restricted representation of $\mathfrak g$. Therefore any of its irreducible representation is one-dimensional with trivial action of the augmentation ideal of $u(\mathfrak g)$. Hence the augmentation ideal of $u(\mathfrak g)$ is nilpotent and $u(\mathfrak g)$ is local. 
\end{proof}

\section{Examples}\label{exm}
In the following, suppose that $k$ has characteristic $p\neq 0$. We give two parametric families of connected, non-cocommutative, non-local, $p^3$-dimensional Hopf algebras, where all the primitive elements are nilpotent. In these cases, the subalgebra generated by $H_1$ is local, yet $H$ is non-local. This ensures the necessity of cocommutativity in our main theorem [Theorem A].  
\begin{example}
Let $\sigma,\lambda,\mu\in k$ so that $\sigma^p=\sigma,\lambda\sigma=(1-\sigma)\mu=0$, and $A$ be the $k$-algebra generated by elements $x,y,z$ subject to the following relations 
\begin{gather*}
[x,y]=0,\ [x,z]=\sigma x,\ [y,z]=(1-\sigma)y,\\
x^p=0,\ y^p=0,\ z^p-z=\lambda x+\mu y.
\end{gather*}
Then $A$ becomes a connected Hopf algebra via
\begin{gather*}
\e(x)=0,\ \Delta(x)=x\otimes 1+1\otimes x,\ S(x)=-x,\\
\e(y)=0,\ \Delta(y)=y\otimes 1+1\otimes y,\ S(y)=-y,\\
\e(z)=0,\ \Delta(z)=z\otimes 1+1\otimes z+x\otimes y,\ S(z)=-z+xy.
\end{gather*}
We denote it as $A(\sigma,\lambda,\mu)$.
\begin{example}
Let $B$ be the $k$-algebra generated by elements $x,y,z$ subject to the following relations 
\begin{gather*}
[x,y]=0,\ [x,z]=x+\sigma y,\ [y,z]=0,\\
x^p=y,\ y^p=0,\ z^p=z,
\end{gather*}
where $\sigma\in k$. Then $B$ becomes a connected Hopf algebra via
\begin{gather*}
\e(x)=0,\ \Delta(x)=x\otimes 1+1\otimes x,\ S(x)=-x,\\
\e(y)=0,\ \Delta(y)=y\otimes 1+1\otimes y,\ S(y)=-y,\\
\e(z)=0,\ \Delta(z)=z\otimes 1+1\otimes z+(x+\sigma y)\otimes y,\ S(z)=-z+(x+\sigma y)y.
\end{gather*}
We denote it as $B(\sigma)$.
\end{example}
\end{example}

\begin{remark}
It is easy to argue that $A$ and $B$ are non-local by contradiction. Suppose that $B$ is local. By Lemma \ref{2duality}(iii), there exists an integer $n$ so that $(B^+)^n=0$. But $z^{p^n}=z\neq 0$ for any $n\ge 0$ and $z\in B^+$, a contradiction. The argument for $A$ being non-local is similar.
\end{remark}

\end{document}